\documentclass[journal,compsoc,twoside]{IEEEtran}

\usepackage{etex} 
\usepackage{amssymb}
\usepackage{graphicx}
\usepackage{multirow}
\usepackage{epstopdf}
\usepackage{algorithmic}
\usepackage[labelfont=bf]{caption}
\usepackage{subcaption}
\usepackage{algorithm}
\usepackage{lineno}
\usepackage[cmex10]{amsmath}
\usepackage{amsthm}
\usepackage{float}
\usepackage{nicefrac}
\usepackage{mathtools}
\usepackage{pgfplots}
\usepackage{pgfplotstable}
\usepackage{pifont}
\usepackage{booktabs}
\usepackage{url}
\usepackage{tikz}
\usepackage[utf8]{inputenc}
\usepackage[english]{babel}
\usepackage{todonotes}
\usepackage{soul}
\usepackage{stmaryrd}
\usepackage{makecell}
\usepackage[inline]{enumitem}
\usetikzlibrary{decorations.markings, decorations.pathreplacing,patterns,fit,arrows}
\usetikzlibrary{automata, positioning}
\usepgfplotslibrary{groupplots}
\usepackage[normalem]{ulem}
\usepackage[flushleft]{threeparttable}

\pgfplotsset{compat=1.13,
  log x ticks with fixed point/.style={
      xticklabel={
        \pgfkeys{/pgf/fpu=true}
        \pgfmathparse{exp(\tick)}%
        \pgfmathprintnumber[fixed relative, precision=3]{\pgfmathresult}
        \pgfkeys{/pgf/fpu=false}
      }
  },
  log y ticks with fixed point/.style={
      yticklabel={
        \pgfkeys{/pgf/fpu=true}
        \pgfmathparse{exp(\tick)}%
        \pgfmathprintnumber[fixed relative, precision=3]{\pgfmathresult}
        \pgfkeys{/pgf/fpu=false}
      }
  }
}

\pgfplotsset{compat=1.13,
	TC_Bri_1 plot/.style={mark=*, mark options={color=blue}},
	arith plot/.style={mark=x, mark options={color=red}},
	TC_Bri_2 plot/.style={mark=square*, mark options={color=orange}},
	bar plot/.style={ybar, fill=white!80!green, mark options={color=green}},
}
\pgfplotsset{compat=1.13,
	TC_Bri_t0 plot/.style={mark=x, mark options={color=blue}},
	TC_Bri_t1 plot/.style={mark=*, mark options={color=black}},
	TC_Bri_E plot/.style={mark=o, mark options={color=red}},
	TC_Bri_t-1 plot/.style={mark=square*, mark options={color=orange}},
	TC_Bri_t-2 plot/.style={mark=diamond, mark options={color=green}},
}

\newtheorem{lemma}{Lemma}
\newtheorem{theorem}{Theorem}

\newtheorem{definition}{Definition}

\DeclarePairedDelimiter\floor{\lfloor}{\rfloor}

\newcommand{\iter}[1]{^{\left(#1\right)}}
\newcommand{\func}[2]{#1{\left(#2\right)}}
\newcommand{\pnorm}[1]{\left\lVert#1\right\rVert_p}
\newcommand{\twonorm}[1]{\left\lVert#1\right\rVert_2}
\newcommand{\infnorm}[1]{\left\lVert#1\right\rVert_\infty}

\newcommand{\abs}[1]{\left|#1\right|}
\newcommand{\natPosNum}{\mathbb{N}_{>0}}

\newcommand\yup{\ding{52}}
\newcommand\nope{\ding{56}}

\newcommand\architect{\textsc{architect}}

\newcommand{\revision}[1]{#1}

\begin{document}

	\title{Digit Stability Inference for Iterative Methods Using Redundant Number Representation}
	
	\author{
		He~Li,~\IEEEmembership{Student~Member,~IEEE},
		Ian~McInerney,~\IEEEmembership{Student~Member,~IEEE},
		James~J.~Davis,~\IEEEmembership{Member,~IEEE},
		and~George~A.~Constantinides,~\IEEEmembership{Senior~Member,~IEEE}%
		\IEEEcompsocitemizethanks{
			\IEEEcompsocthanksitem The authors are with the Department of Electrical and Electronic Engineering, Imperial College London, London, SW7 2AZ, United Kingdom.
			E-mail: \{h.li16, i.mcinerney17, james.davis, g.constantinides\}@imperial.ac.uk.
		}
	}

	\markboth{IEEE Transactions on Computers}{Li et al.: Digit Stability Inference for Iterative Methods Using Redundant Number Representation}
	
	\IEEEtitleabstractindextext{
    	\begin{abstract}

            In our recent work on iterative computation in hardware, we showed that arbitrary-precision solvers can perform more favorably than their traditional arithmetic equivalents when the latter's precisions are either under- or over-budgeted for the solution of the problem at hand.
            Significant proportions of these performance improvements stem from the ability to infer the existence of identical most-significant digits between iterations.
            This technique uses properties of algorithms operating on redundantly represented numbers to allow the generation of those digits to be skipped, increasing efficiency.
            It is unable, however, to guarantee that digits will stabilize, i.e., never change in any future iteration.
            In this article, we address this shortcoming, using interval and forward error analyses to prove that digits of high significance will become stable when computing the approximants of systems of linear equations using stationary iterative methods.
            We formalize the relationship between matrix conditioning and the rate of growth in most-significant digit stability, using this information to converge to our desired results more quickly.
            Versus our previous work, an exemplary hardware realization of this new technique achieves an up-to 2.2$\times$ speedup in the solution of a set of variously conditioned systems using the Jacobi method.
    		
        \end{abstract}
        
    }
    
    \maketitle
    
    \section{Introduction \& Motivation}
        
        \revision{Many scientific, optimization, and machine learning applications require the solution of systems of linear equations~\cite{Olshanskii2014IterativeMethods}}.
        Stationary iterative methods such as Gauss-Seidel, Jacobi, and successive over-relaxation are popular ways to convert such an $N$-dimensional system, $\boldsymbol{A}\boldsymbol{x} = \boldsymbol{b}$, into a linear fixed-point iteration.
        These all take the form $\boldsymbol{x}\iter{k+1} = \func{f}{\boldsymbol{x}\iter{k}}$, where $f : \mathbb{R}^N \to \mathbb{R}^N$ is a computable real function.
        
        When solving such systems conventionally, cases arise where low-magnitude perturbations cause large numbers of digits to change between iterations via carry propagation.
        Fig.~\ref{fig:digit_gen_nonredundant} exemplifies this for the toy iteration
        \begin{equation*}
        	x\iter{k + 1} = \nicefrac{5}{4} - \nicefrac{1}{4} \cdot x\iter{k}
        \end{equation*}
        computed from $x\iter{0} = 0$ with nonredundant radix-10 number representation.
        Here, the method causes oscillations in approximants around the true result, $x^* = 1$.
        Although the absolute algorithm residue $\abs{x\iter{k} - x^*}$ decreases monotonically as $k \to \infty$, digits across approximants never stabilize.
        
        \begin{figure}
        	\centering
        	\begin{subfigure}[t]{0.1\linewidth}
        		\centering
        		\begin{tikzpicture}[yscale=-1,inner sep=0.5mm,scale=0.6]
		
	\foreach \y in {1,...,6}
		\node at (-1.5,\y) {\vphantom{$1$}\smash{$x\iter{\y}{:}$}};
	
\end{tikzpicture}
        	\end{subfigure}
        	\begin{subfigure}[t]{0.02\linewidth}
        	    \hfill
        	\end{subfigure}
        	\begin{subfigure}[t]{0.4\linewidth}
        		\centering
        		\begin{tikzpicture}[yscale=-0.6,inner sep=0.5mm, xscale=0.5]
	
	\newcommand\g{\color{black!25}}

	\foreach \y in {0,...,5}
		\node(p\y) at (0.5,\y) {\vphantom{$1$}$.$};

	\node(x01) at (0,0) {$1$};
	\node(x11) at (1,0) {$2$};
	\node(x21) at (2,0) {$5$};
	\node(x31) at (3,0) {$0$};
	\node(x41) at (4,0) {$0$};
	\node(x51) at (5,0) {$0$};
	\node(x61) at (6,0) {$0$};

	\node(x02) at (0,1) {$1$};
	\node(x12) at (1,1) {$0$};
	\node(x22) at (2,1) {$1$};
	\node(x32) at (3,1) {$5$};
	\node(x42) at (4,1) {$6$};
	\node(x52) at (5,1) {$2$};
	\node(x62) at (6,1) {$5$};

	\node(x03) at (0,2) {$0$};
	\node(x13) at (1,2) {$9$};
	\node(x23) at (2,2) {$9$};
	\node(x33) at (3,2) {$6$};
	\node(x43) at (4,2) {$0$};
	\node(x53) at (5,2) {$9$};
	\node(x63) at (6,2) {$3$};

	\node(x04) at (0,3) {$1$};
	\node(x14) at (1,3) {$0$};
	\node(x24) at (2,3) {$0$};
	\node(x34) at (3,3) {$0$};
	\node(x44) at (4,3) {$9$};
	\node(x54) at (5,3) {$7$};
	\node(x64) at (6,3) {$6$};

	\node(x05) at (0,4) {$0$};
	\node(x15) at (1,4) {$9$};
	\node(x25) at (2,4) {$9$};
	\node(x35) at (3,4) {$9$};
	\node(x45) at (4,4) {$7$};
	\node(x55) at (5,4) {$5$};
	\node(x65) at (6,4) {$5$};

	\node(x06) at (0,5) {$1$};
	\node(x16) at (1,5) {$0$};
	\node(x26) at (2,5) {$0$};
	\node(x36) at (3,5) {$0$};
	\node(x46) at (4,5) {$0$};
	\node(x56) at (5,5) {$6$};
	\node(x66) at (6,5) {$1$};

\end{tikzpicture}
        		\caption{Nonredundant form.}
        		\label{fig:digit_gen_nonredundant}
        	\end{subfigure}
        	\begin{subfigure}[t]{0.02\linewidth}
        	    \hfill
        	\end{subfigure}
        	\begin{subfigure}[t]{0.4\linewidth}
        		\centering
        		\begin{tikzpicture}[yscale=-0.6,inner sep=0.5mm,xscale=0.5]
	
	\newcommand\g{\color{black!25}}

	\foreach \y in {0,...,0}
		\node(p\y) at (0.5,\y) {\vphantom{$1$}$.$};
	\foreach \y in {1,...,5}
		\node(p\y) at (0.5,\y) {\vphantom{$1$}\g$.$};

	\node(x01) at (0,0) {$1$};
	\node(x11) at (1,0) {$2$};
	\node(x21) at (2,0) {$5$};
	\node(x31) at (3,0) {$0$};
	\node(x41) at (4,0) {$0$};
	\node(x51) at (5,0) {$0$};
	\node(x61) at (6,0) {$0$};

	\node(x02) at (0,1) {\g$1$};
	\node(x12) at (1,1) {$0$};
	\node(x22) at (2,1) {$1$};
	\node(x32) at (3,1) {$5$};
	\node(x42) at (4,1) {$6$};
	\node(x52) at (5,1) {$2$};
	\node(x62) at (6,1) {$5$};

	\node(x03) at (0,2) {\g$1$};
	\node(x13) at (1,2) {\g$0$};
	\node(x23) at (2,2) {$0$};
	\node(x33) at (3,2) {$\overline{4}$};
	\node(x43) at (4,2) {$1$};
	\node(x53) at (5,2) {$\overline{1}$};
	\node(x63) at (6,2) {$3$};

	\node(x04) at (0,3) {\g$1$};
	\node(x14) at (1,3) {\g$0$};
	\node(x24) at (2,3) {\g$0$};
	\node(x34) at (3,3) {$1$};
	\node(x44) at (4,3) {$0$};
	\node(x54) at (5,3) {$\overline{3}$};
	\node(x64) at (6,3) {$6$};

	\node(x05) at (0,4) {\g$1$};
	\node(x15) at (1,4) {\g$0$};
	\node(x25) at (2,4) {\g$0$};
	\node(x35) at (3,4) {\g$0$};
	\node(x45) at (4,4) {$\overline{3}$};
	\node(x55) at (5,4) {$5$};
	\node(x65) at (6,4) {$5$};

	\node(x06) at (0,5) {\g$1$};
	\node(x16) at (1,5) {\g$0$};
	\node(x26) at (2,5) {\g$0$};
	\node(x36) at (3,5) {\g$0$};
	\node(x46) at (4,5) {\g$0$};
	\node(x56) at (5,5) {$6$};
	\node(x66) at (6,5) {$1$};

\end{tikzpicture}
        		\caption{Redundant form.}
        		\label{fig:digit_gen_redundant}
        	\end{subfigure}
        	\caption{
        		Approximants of toy iteration $x\iter{k+1} = \nicefrac{5}{4} - \nicefrac{1}{4}\cdot x\iter{k}$ with both nonredundant and redundant radix-10 number representation.
        		In (\subref{fig:digit_gen_redundant}), digits in gray are known to have stabilized, and thus do not need to be recomputed.
        		Digits with over-bars represent negative values: $\overline{i} = -i$.
        	}
        	\label{fig:digit_gen}
        \end{figure}
        
        By introducing redundancy into our number representation, we can prevent the occurrence of this scenario.
        \revision{Fig.~\ref{fig:digit_gen_redundant} shows the same example, also radix-10, but now with digits able to be selected from the \emph{redundant digit set} $\left\{-9,-8,\cdots,8,9\right\}$.
        Here, less-significant digits (LSDs) can be used to correct errors that were previously introduced by digits of higher significance, further allowing those most-significant digits (MSDs) to be declared stable.}
        The computation of these digits---shown in gray---can thus be avoided, increasing computational efficiency.
        
        In our previous work, \architect{}, we introduced the first hardware architecture capable of computing solutions of systems of linear equations to arbitrary accuracy~\cite{ARCHITECT,li2018digit,ARCHITECT_TVLSI}.
        \architect{} uses redundant number representation to allow approximants to be computed from MSD first with \emph{online arithmetic}, enabling earlier approximants to be refined as needed.
        With the knowledge that some $D$ MSDs are common to approximants $k$ and $k + 1$, \architect{} is able to deduce the number that will also appear in approximant $k + 2$.
        Since that number is always smaller than $D$, however, this technique is unfortunately unable to infer digit stability.
        
        Also using online arithmetic, Ercegovac's E-method produces the digits of its results from MSD first, one more per iteration~\cite{ercegovac1977general}.
        As exemplified in Fig.~\ref{plt:intro}, this technique therefore enables the inference of digit stability.
        The E-method, however, is a specialized Jacobi iteration and imposes strict conditions on its inputs: particularly a well conditioned $\boldsymbol{A}$.
        
        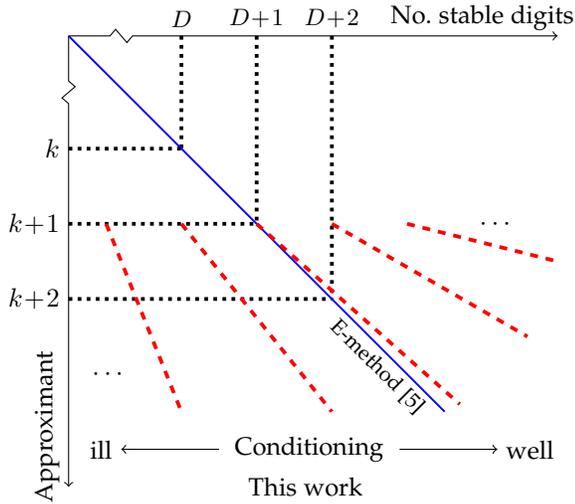
\begin{figure}
			\centering
			\begin{tikzpicture}[yscale=-1.0, xscale=1.0]

\def\xz{-0.3}
\def\xa{1.0}
\def\xn{3.3}
\def\xb{4.3}
\def\barh{-2.5 mm}
\def\bari{15 mm}
\def\barb{19 mm}
\def\barj{22 mm}
\def\barc{25 mm}
\tikzset{  
	brace/.style={decorate, decoration={brace,amplitude=1mm}},
}

\draw[->] (-0.5,-0.5) -- (0.05,-0.5) decorate[decoration=zigzag,segment length=3mm]{--(0.45,-0.5)} -- (6.0,-0.5);
\node(D1)[anchor=south] at (1,-0.5) {\small{$D$}};
\node(D1)[anchor=south] at (2,-0.5) {\small{$D\!+\!1$}};
\node(D1)[anchor=south] at (3,-0.5) {\small{$D\!+\!2$}};
\draw[->] (-0.5, -0.5) -- (-0.5,0.1) decorate[decoration=zigzag,segment length=3mm]{--(-0.5,0.5)} -- (-0.5, 1.0) node[left, rotate=0] {$k$} -- (-0.5, 2.0) node[left, rotate=0] {$k\!+\!1$} -- (-0.5, 3.0) node[left, rotate=0] {$k\!+\!2$} -- (-0.5, 4.7) node[above, rotate=90] {Approximant} -- (-0.5, 5.5);
\node[anchor=south] at (5.0,-0.45) {No. stable digits};

\newcommand\braceheight{-0.6}

\draw[-, black,  dotted, line width=0.50mm] (2.0, -0.5) to (2.0,2.0);
\draw[-, black,  dotted, line width=0.50mm] (-0.5, 2.0) to (2.0,2.0);
\draw[-, black,  dotted, line width=0.50mm] (1.0, -0.5) to (1.0,1.0);
\draw[-, black,  dotted, line width=0.50mm] (-0.5, 1.0) to (1.0,1.0);
\draw[-, black,  dotted, line width=0.50mm] (3.0, -0.5) to (3.0,3.0);
\draw[-, black,  dotted, line width=0.50mm] (-0.5, 3.0) to (3.0,3.0);

\draw[-, blue, line width=0.25mm] (-0.5, -0.5) to (4.5,4.5);

\draw[-, red, dashed, line width=0.50mm] (2.0, 2.0) to (4.7,4.4);
\draw[-, red, dashed, line width=0.50mm] (3.0, 2.0) to (5.6,3.5);
\node[black, inner sep=0, rotate = 0] (5.2, 2.0) at (5.2,2.0) {$\cdots$};
\draw[-, red, dashed, line width=0.50mm] (4.0, 2.0) to (6.0,2.5);

\draw[-, red, dashed, line width=0.50mm] (1.0, 2.0) to (3.0, 4.5);
\draw[-, red, dashed, line width=0.50mm] (0.0, 2.0) to (1.0,4.5);
\node[black, inner sep=0] (0.05,4.0) at (0.05,4.0) {$\cdots$};

\draw[->] (3.9, 5.0) -- (5.2, 5.0) node[right] {well};
\node(Emethod) at (2.7, 5.0) {Conditioning};
\draw[->] (1.5, 5.0) -- (0.2, 5.0) node[left] {ill};

\node(Emethod)[rotate=-45] at (3.6, 3.9) {\footnotesize{E-method~\cite{ercegovac1977general}}};
\node(Ours) at (2.65, 5.5) {This work};

\end{tikzpicture}
			\caption{
			    A sketch of guaranteed digit stability.
			    The E-method produces one new digit of lower significance per iteration; these, whose boundary is represented by the solid blue line, therefore remain stable across all future approximants.
			    With knowledge that approximants $k$ and $k + 1$ share $D$ identical MSDs, our technique is able to infer the numbers of stable digits within the $k + 1$th and all future approximants.
			    As shown by the dashed red lines, these are dependent upon the conditioning of $\boldsymbol{A}$.
		    }
			\label{plt:intro}
    	\end{figure}
        
        In this article, we revisit \architect{}'s MSD elision, combining knowledge of MSDs shared between successive approximants with matrix conditioning to infer digit stability.
        In contrast to the E-method, this work is applicable to any stationary iterative method and, as also shown in Fig.~\ref{plt:intro}, holds for both well and ill-conditioned $\boldsymbol{A}$.
        With particularly well conditioned matrices, we can predict the generation of more than one stable digit per iteration.
        
        We make the following novel contributions in this article:
        \begin{itemize}
            \item
                Using interval and forward error analyses, a theorem for the rate of stable MSD growth within the approximants produced by any stationary iterative method.
            \item
                Theoretical comparison of our proposal versus existing methods allowing the skipping of MSD calculation.
            \item
                An exemplary hardware implementation of our proposal using the Jacobi method.
            \item
                Empirical performance comparisons against the state-of-the-art arbitrary-precision iterative solver.
                For the solution of a set of representative linear equations, we achieve speedups of 2.0--2.2$\times$ over this prior work.
        \end{itemize}
		
	\section{Background}

		\subsection{\revision{Redundant Number Representation}}
		\revision{In a redundant number system, the representation associated with a value is not unique, i.e., the same value can be encoded in more than one way.
	    The most widely used of the redundant systems is the class of \emph{symmetric signed-digit} number representations~\cite{lu2004arithmetic}, originally conceived for the purpose of performing carry-free addition~\cite{avizienis1961signed}.
        Here, digits can take any value from within the set
        \begin{equation*}
        	S \coloneqq \left\{-\gamma,-\gamma+1,\cdots,\gamma-1,\gamma\right\},
        \end{equation*}
        where $\nicefrac{r}{2} \leq \gamma \leq r-1$ for some radix $r$.
        When $\gamma = r-1$, the digit set $S$ is said to be \emph{maximally redundant}.
        Henceforth, we assume the use of a symmetric, maximally redundant digit set, and our example implementation uses radix-2 number representation of this form.
        Our results could be extended to asymmetric ($\min{S} \neq -\max{S}$) and non-maximally redundant digit sets if required.}
		
		\subsection{Hardware Applications of Redundancy}

            \revision{The performance of many custom hardware systems is predominantly dependent upon the speed of their underlying arithmetic operators~\cite{wang2019deep}.}
            When these employ conventional, nonredundant number representations, carry propagation is often the primary factor determining their latency.
            The introduction of redundancy, however, often allows execution times to be shortened due to the reduction---and sometimes complete elimination---of carry chains~\cite{jaberipur2017redundant}.
            Interest in the acceleration of arithmetic circuits using redundant number systems is growing.
            Signed-digit representations have been used within high-radix adders~\cite{timarchi2015maximally} and dividers~\cite{amanollahi2016energy}, and constant-vector multipliers~\cite{fan2016improved}, to improve performance and reduce power consumption versus their conventional equivalents.
            Fast multipliers using signed-digit representation during their partial product generation~\cite{cui2015modified} and reduction~\cite{kaivani2015floating} steps have also been proposed.
            
            Similarly to the E-method~\cite{ercegovac1977general}, the work we describe herein uses redundancy in order to infer digit stability within iterative algorithms.
            In contrast to that technique, however, ours is less restrictive and more widely applicable.
		
		\subsection{Online Arithmetic Essentials}
        \label{sec:oa}
	        
	        The \emph{de facto} standard for MSD-first calculation is Ercegovac's \emph{online arithmetic}~\cite{OAbook}.
	        An important characteristic of online operators is that of \emph{online delay}, typically denoted $\delta$.
	        Classical digit-serial online operators produce output digits at the same rate as they consume them, but delayed by a fixed number of digits: $\delta$.
	        When operators are chained to form a datapath, its overall online delay is the summation of operators' delays through the longest path~\cite{fpt16}.
	        
	        Online delay is typically considered to be a limitation in terms of throughput, thus effort has been made to reduce it through the use of composite online functions~\cite{adharapurapu2004composite,ercegovac2007digit}, multioperand operators~\cite{joseph2018algorithms,villalba2011radix}, and high radices~\cite{joseph2016design}.
	        For \architect{}, we showed that online delay could also be used to infer the presence of identical MSDs within iterative computations~\cite{li2018digit,ARCHITECT_TVLSI}.
			Since datapaths composed of online operators compute from MSD first and outputs begin to be generated $\delta$ digits after input digits are consumed, an output's first $D$ digits are wholly dependent upon its inputs' first $D+\delta$ digits.
			Since iterative methods' inputs are its previously generated outputs, this allows us to guarantee that, if approximants $k$ and $k+1$ are equal in their first $D$ MSDs, approximant $k+2$ will have $D-\delta$ MSDs in common with both when computed using online arithmetic.
			In this article, we prove that MSDs can be declared identical not just for \emph{one} approximant, but across \emph{all} future approximants.
            
        \subsection{Stationary Iterative Methods}
            
            In numerical linear algebra, a straightforward way to solve a system $\boldsymbol{A}\boldsymbol{x} = \boldsymbol{b}$ is to transform it into a linear fixed-point iteration of the form
            \begin{equation}
        	    \boldsymbol{M}\boldsymbol{x}\iter{k+1} = \boldsymbol{N}\boldsymbol{x}\iter{k} + \boldsymbol{b}
        	    \label{eqn:iter_exact}
        	\end{equation}
        	with $\boldsymbol{A} = \boldsymbol{M} - \boldsymbol{N}$ and $\boldsymbol{M}$ non-singular~\cite{higham2002accuracy}%
        	\footnote{\revision{
        	    If desired, explicit preconditioning can be applied by substituting a preconditioning matrix $\boldsymbol{R}$ for $\boldsymbol{M}$, resulting in an alternative stationary iterative method~\cite{evans1992modified}.
        	    Implicit preconditioning is beyond the scope of this article, but could be incorporated if it can be expressed in terms of standard online arithmetic operators~\cite{ercegovac1977general}.
            }}.
        	Defining iteration matrix $\boldsymbol{G} = \boldsymbol{M}^{-1}\boldsymbol{N}$, \eqref{eqn:iter_exact} can also be written as
        	\begin{equation}
        	    \boldsymbol{x}\iter{k+1} = \boldsymbol{G}\boldsymbol{x}\iter{k} + \boldsymbol{M}^{-1}\boldsymbol{b}.
        	    \label{eqn:iter_exact_g}
        	\end{equation}
        	\revision{Achievement of convergence requires that $\boldsymbol{G}$'s spectral radius $\rho(\boldsymbol{G}) < 1$.
        	Such \emph{stationary iterative methods} are widely used in the approximate solution of nonlinear~\cite{macklin2019non}, differential~\cite{koh2019pricing}, and integral equations~\cite{yuan2019overview}.}
        	They also play a significant role in multigrid theory; multigrid methods commonly serve as preconditioners for many other iterative algorithms~\cite{barrett1994templates}.
            \revision{For scenarios in which high-precision results are required, mixed-precision methods enabling performant and efficient implementation have been proposed~\cite{buttari2008using}.
            In contrast to standard approaches, we adopt an MSD-first arbitrary-precision computation paradigm enabling iterative refinement limited only by memory capacity~\cite{ARCHITECT}.}
            
            The work we present in this article applies to any method of the form in \eqref{eqn:iter_exact}.
            While we use stationary iterative methods as accessible examples for our analysis, our proposal could be extended to nonlinear fixed-point iterations.

	\section{Preliminaries \& Notation}
        
        For the remainder of this article, we assume the use of a fixed-point radix-$r$ symmetric signed-digit number representation system with maximal redundancy.
        
        A scalar is denoted by a normal symbol $x$.
        For convenience, we assume that all redundantly represented numbers have $\abs{x} < 1$ and can be expressed as $x = \sum_{i = 1}^D x_i r^{-i}$, where $x_i$ is the $i$th MSD of $D$-digit $x$.
        
        A vector is represented by a bold symbol $\boldsymbol{x}$, with its $j$th element denoted $x_j$.
        Where a vector is composed of signed-digit numbers, $x_{ji}$ is the $i$th MSD of the $j$th element of $\boldsymbol{x}$.
        
        An approximant of an iterative method at iteration $k \in \natPosNum$ is denoted $\boldsymbol{x}\iter{k}$, while its exact result is $\boldsymbol{x}^*$.
        The residue of an iterative method at iteration $k$ is $\boldsymbol{s}\iter{k} = \boldsymbol{x}\iter{k} - \boldsymbol{x}^*$.
        
        A matrix is represented by a bold capital symbol $\boldsymbol{X}$, and the $p$-norm of either a matrix or a vector is given by $\pnorm{\bullet}$.
	
	\section{Digit Stability Inference}
	\label{sec:theory}
        
        Assume that a stationary iterative method is used to solve a linear system $\boldsymbol{A}\boldsymbol{x} = \boldsymbol{b}$.
        \revision{Further assume that the inequality $\infnorm{\boldsymbol{G}} < 1$ holds}%
        \footnote{\revision{
            We adopt the infinity-norm in the analysis that follows since digit stability is ensured through bounds on worst-case perturbations of $\boldsymbol{x}\iter{k}$.
            In the common case of Hermitian matrices, $\func{\rho}{\boldsymbol{G}} = \twonorm{\boldsymbol{G}}$ and $\twonorm{\boldsymbol{G}} \leq \infnorm{\boldsymbol{G}}$, thus a bound on $\infnorm{\boldsymbol{G}}$ corresponds to a bound on $\func{\rho}{\boldsymbol{G}}$~\cite{Horn}.
            Finally note that, although we present our analysis in a general setting, its application is intended for methods where $\infnorm{\boldsymbol{G}}$ is readily computable, such as Jacobi.
        }}.
        If approximants to $\boldsymbol{x}^*$ are vectors with digits selected from a symmetric maximally redundant signed-digit set, knowledge of the number of identical MSDs in any two successive approximants $\hat{k} - 1$ and $\hat{k}$ allows us to declare that subsets of MSDs in all approximants $k \geq \hat{k}$ will never change.
		The key steps in the derivation that follows are:
		\begin{enumerate}
		    \item
		        \textbf{Lemma~\ref{lem:ResidueBound}}: If it is known that $D$ MSDs of successive approximants' elements are identical, we can bound the magnitude of the algorithm residue based on $D$ and $\boldsymbol{G}$.
		    \item
		        \textbf{Lemma~\ref{lem:ConstantExistence}}: Given a particular residue bound, we prove that a quantity of the current and future approximants' MSDs can never change.
		    \item
		        \textbf{Theorem~\ref{thm:StableDigits}}: Bringing Lemmas~\ref{lem:ResidueBound} and \ref{lem:ConstantExistence} together, we infer the minimum number of permanently identical MSDs per approximant based on $D$ and $\boldsymbol{G}$.
		\end{enumerate}
    		
    	Let us begin by formally defining the meaning of digit stability within the approximants of an iterative algorithm.
    	
        \begin{definition}[Digit stability]
    	\label{def:StableDigits}
    	    
    	    \revision{The $D$ MSDs of an approximant $\hat{k}$ are said to be stable iff}
    	    \begin{equation*}
    	        {x}\iter{k}_i = {x}\iter{\hat{k}}_i \quad \forall k > \hat{k}~\forall i \in \left\{1,2,\cdots,D\right\}.
    	    \end{equation*}
    	    
    	\end{definition}
        
        \revision{Our choice of number system means that we can append digits to a number $x$ to form a new number, $\tilde{x}$, representing any value within a symmetric interval around $x$.}
        We call such numbers \emph{consistent} in the values they represent.
        
        \begin{definition}[Digit consistency]
    	\label{def:ConsistentDigits}
    	    
            Let $x$ be a number composed of $D$ digits selected from a symmetric maximally redundant signed-digit set.
    	    Further let $y$ be a second number, similarly constructed, comprising any finite number of digits.
    	    $y$ is said to be consistent with $x$ iff 
    	    \begin{equation*}
    	        y \in \left(x - r^{-D},~x + r^{-D}\right).
    	    \end{equation*}
    	\end{definition}
        
        \begin{lemma}[Representation interval]
        \label{lem:redundantInterval}
    		
    		Let $x$ be a $D$-digit number.
    		\revision{If additional digits are appended to $x$ to form a new number, $\tilde{x}$, then $\tilde{x}$ is consistent with $x$.}
    		
    	\end{lemma}
    	
    	\begin{proof}
    	    
    	    By definition,
    	    \begin{equation*}
    	        x = \sum_{i=1}^{D} x_i r^{-i}.
    	    \end{equation*}
    	   \revision{Since $\tilde{x}$ contains $\tilde{D} > D$ digits, with its $D$ MSDs the same as those in $x$,
    	    \begin{align*}
    	        \tilde{x} &= \sum_{i=1}^{D} x_i r^{-i} + \sum_{i=D+1}^{\tilde{D}} \tilde{x_i} r^{-i}
    	        \\
    	                &= x + \sum_{i=D+1}^{\tilde{D}} \tilde{x}_i r^{-i}.
    	    \end{align*}}

    	    The digit extrema in our number system are $-\left(r - 1\right)$ and $r - 1$.
    	    We can thus deduce that
    	    \revision{\begin{align*}
    	        \tilde{x} &\in \left[x - \sum_{i=D+1}^{\tilde{D}} \left(r-1\right)r^{-i},~x + \sum_{i=D+1}^{\tilde{D}} \left(r-1\right)r^{-i} \right]
    		    \\
    		            &= \left[x - r^{-D} + r^{-\tilde{D}},~x + r^{-D} - r^{-\tilde{D}}\right]
    		    \\
    		            &\subset \left(x - r^{-D},~x + r^{-D}\right),
    		\end{align*}
    		and so, per Definition~\ref{def:ConsistentDigits}, $\tilde{x}$ is consistent with $x$.}
    	\end{proof}
    	
    	Suppose now that we know---via runtime digit-by-digit comparison---that some $D$ MSDs within successive approximants $k$ and $k + 1$ are identical. Given particular iteration matrix conditioning, we can bound the algorithm residue for approximant $k + 1$.
        
        \begin{lemma}[Residue bound]
        \label{lem:ResidueBound}

    	    If the elements of $\boldsymbol{x}\iter{k}$ and $\boldsymbol{x}\iter{k + 1}$ share a minimum of $D$ identical MSDs, then
            \begin{equation*}
    	        \infnorm{\boldsymbol{s}\iter{k + 1}} < \frac{2\infnorm{\boldsymbol{G}}}{1 - \infnorm{\boldsymbol{G}}} r^{-D}.
            \end{equation*}
        
        \end{lemma}
        
        \begin{proof}

            Manipulation of \eqref{eqn:iter_exact} allows us to deduce that
            \begin{align*}
				\boldsymbol{M}{\left(\boldsymbol{x}\iter{k + 1} - \boldsymbol{x}\iter{k}\right)}    &= \left(\boldsymbol{N} - \boldsymbol{M}\right)\boldsymbol{x}\iter{k} + \boldsymbol{A}\boldsymbol{x}^*
				\\
				\boldsymbol{A}\boldsymbol{s}\iter{k}                                                &= \boldsymbol{M}{\left(\boldsymbol{x}\iter{k} - \boldsymbol{x}\iter{k + 1}\right)}.
            \end{align*}
            Given that $\boldsymbol{A}^{-1} = \sum_{i=0}^{\infty}\boldsymbol{G}^i\boldsymbol{M}^{-1}$~\cite{li2018digit}, we therefore have
            \begin{equation*}
                \boldsymbol{s}\iter{k} = \sum_{i=0}^{\infty}\boldsymbol{G}^i{\left(\boldsymbol{x}\iter{k} - \boldsymbol{x}\iter{k + 1}\right)}.
            \end{equation*}
            Taking norms and recalling that $\infnorm{\boldsymbol{G}} < 1$,
            \begin{align}
                \infnorm{\boldsymbol{s}\iter{k}}    &\leq \infnorm{\sum_{i=0}^{\infty}\boldsymbol{G}^i}\infnorm{\boldsymbol{x}\iter{k} - \boldsymbol{x}\iter{k + 1}}
                \nonumber\\
                                                    &\leq \frac{1}{1 - \infnorm{\boldsymbol{G}}}\infnorm{\boldsymbol{x}\iter{k} - \boldsymbol{x}\iter{k + 1}}.
                \label{eqn:k_diffs_2_residue}
            \end{align}
            
            Let $j$ be the index for which $x_j\iter{k}$ and $x_j\iter{k + 1}$ are the successive elements sharing the fewest identical MSDs.
            We define the number of contiguous MSDs shared by the $j$th elements as $D$.
            From Lemma~\ref{lem:redundantInterval} we know that
            \begin{equation*}
                x_j\iter{k} \in \left(\sum_{i = 1}^D x_{ji}\iter{k} r^{-i} - r^{-D},~\sum_{i = 1}^D x_{ji}\iter{k} r^{-i} + r^{-D}\right)
            \end{equation*}
            and
            \begin{equation*}
                x_j\iter{k + 1} \in \left(\sum_{i = 1}^D x_{ji}\iter{k + 1} r^{-i} - r^{-D},~\sum_{i = 1}^D x_{ji}\iter{k + 1} r^{-i} + r^{-D}\right).
            \end{equation*}
            Since $x_{ji}\iter{k} = x_{ji}\iter{k + 1}~\forall i \in \left\{1,2,\cdots,D\right\}$, we find that
            \begin{equation*}
    	        \abs{x_j\iter{k} - x_j\iter{k + 1}} < 2r^{-D},
    	    \end{equation*}
            giving a bound on the vector norm of
            \begin{equation}
                \infnorm{\boldsymbol{x}\iter{k} - \boldsymbol{x}\iter{k + 1}} < 2r^{-D}.
                \label{eqn:vector_diff_K}
            \end{equation}
            
            Transformation of \eqref{eqn:iter_exact_g} reveals that
            \begin{align*}
                \boldsymbol{x}\iter{k + 1}  &= \boldsymbol{G}\boldsymbol{x}\iter{k} + \boldsymbol{M}^{-1}\boldsymbol{A}\boldsymbol{x}^*
                \\
                &= \boldsymbol{G}\boldsymbol{x}\iter{k} + \left(\boldsymbol{I} - \boldsymbol{G}\right)\boldsymbol{x}^*
                \\
                \boldsymbol{s}\iter{k + 1}  &= \boldsymbol{G}\boldsymbol{s}\iter{k}.
            \end{align*}
            Taking norms,
            \begin{equation}
                \infnorm{\boldsymbol{s}\iter{k + 1}} \leq \infnorm{\boldsymbol{G}}\infnorm{\boldsymbol{s}\iter{k}},
                \label{eqn:norm_residue_decrease}
            \end{equation}
            which, when combined with \eqref{eqn:k_diffs_2_residue}, results in
            \begin{equation*}
                \infnorm{\boldsymbol{s}\iter{k + 1}} \leq \frac{\infnorm{\boldsymbol{G}}}{1 - \infnorm{\boldsymbol{G}}} \infnorm{\boldsymbol{x}\iter{k} - \boldsymbol{x}\iter{k + 1}}.
            \end{equation*}
            Substitution of \eqref{eqn:vector_diff_K} then gives
            \begin{equation}
                \infnorm{\boldsymbol{s}\iter{k + 1}} < \frac{2\infnorm{\boldsymbol{G}}}{1 - \infnorm{\boldsymbol{G}}} r^{-D}.
                \label{eqn:res_bound_G}
            \end{equation}

        \end{proof}
	
	    Given a particular residue bound, our next task is to show that we can guarantee MSD stability within the current and future approximants.
         	
     	\begin{lemma}[Existence of digit stability]
     	\label{lem:ConstantExistence}
    		
    		If the condition
    		\begin{equation}
    		    \infnorm{\boldsymbol{s}\iter{\hat{k}}} < r^{-D}
    		    \label{eqn:bound_shatk}
    		\end{equation}
    		holds, then $x_j^*$ is consistent with the $D - 1$ MSDs of $x_j\iter{k}~\forall k \geq \hat{k}~\forall j$, and these MSDs are stable.
    		
    	\end{lemma}
    	
    	\begin{proof}

    	    Convergence results on the algorithm ensure that there must exist an approximant $\hat{k}$ for which
    	    \begin{equation*}
    	        x_j^* \in \left(x_j\iter{\hat{k}} - r^{-D},~x_j\iter{\hat{k}} + r^{-D}\right) \quad\forall j.
    	    \end{equation*}
    	    From Lemma~\ref{lem:redundantInterval}, we know that $x_j^*$ is consistent with $x_j\iter{\hat{k}}$.
    	    
    	    Through repeated self-substitution of \eqref{eqn:norm_residue_decrease},
            \begin{equation}
                \infnorm{\boldsymbol{s}\iter{k}} \leq \infnorm{\boldsymbol{G}}^{k - \hat{k}}\infnorm{\boldsymbol{s}\iter{\hat{k}}}
                \label{eqn:residue_bound_k}
            \end{equation}
    	    which, given \eqref{eqn:bound_shatk}, means that
    	    \begin{equation*}
                \infnorm{\boldsymbol{s}\iter{k}} < \infnorm{\boldsymbol{G}}^{k - \hat{k}}r^{-D}
            \end{equation*}
            and thus
    	    \begin{equation*}
    	        \abs{s_j\iter{k}} < \infnorm{\boldsymbol{G}}^{k - \hat{k}}r^{-D}\quad\forall j.
    	    \end{equation*}
    	    For approximant $k$, therefore,
    	    \begin{equation*}
    	        x_j^* \in I \coloneqq \left(x_j\iter{k} - \infnorm{\boldsymbol{G}}^{k - \hat{k}}r^{-D},~x_j\iter{k} + \infnorm{\boldsymbol{G}}^{k - \hat{k}}r^{-D}\right) \quad\forall j.
    	    \end{equation*}
    	    
    	    Let us consider how the perturbation of one or more of the $D - 1$ MSDs in any approximant $k \geq \hat{k}$ would affect algorithmic convergence.
    	    Such a perturbation would produce a new interval, $I'$.
    	    If $x_j^* \notin I'$, such a new representation of $x_j$ would be inconsistent with the proof of convergence, thus the $D - 1$ MSDs of $x_j\iter{k}$ must be identical for all $k \geq \hat{k}$.

            Consider an increase of the $D - 1$th MSD by one unit, leading to a representation consistent with any value in
            \begin{multline*}
    	        I' \coloneqq \Big(x_j\iter{\hat{k}} - \infnorm{\boldsymbol{G}}^{k - \hat{k}}r^{-D} + r^{-(D - 1)},
    	        \\
    	        x_j\iter{\hat{k}} + \infnorm{\boldsymbol{G}}^{k - \hat{k}}r^{-D} + r^{-\left(D - 1\right)}\Big) \quad\forall j.
    	    \end{multline*}
    	    Comparing the upper bound of $I$ and the lower bound of $I'$, we have
    	    \begin{align}
    	        \min{I'} - \max{I}  &= -2\infnorm{\boldsymbol{G}}^{k - \hat{k}}r^{-D} + r^{-\left(D - 1\right)} \nonumber
    	        \\
    	                            &= \left(r - 2\infnorm{\boldsymbol{G}}^{k - \hat{k}}\right)r^{-D}. \label{eqn:minmax}
    	    \end{align}
    	    Since $r \geq 2$ and $\infnorm{\boldsymbol{G}} < 1$, \eqref{eqn:minmax} is strictly positive. This means that $I \cap I' = \emptyset$, and thus $x_j^* \notin I'$.
    	    
    	    Clearly, a unit increase of \emph{any} digit in $x\iter{k}_{ji}~\forall i \in \left\{1,2,\cdots,D - 1\right\}$ would lead to $I \cap I' = \emptyset$, violating the algorithm's convergence.
    	    A similar argument can be made for a unit decrease of $x\iter{k}_{ji}~\forall i \in \left\{1,2,\cdots,D - 1\right\}$.
    	    Thus, $x_j^*$ is consistent with the $D - 1$ MSDs of $x_j\iter{k}~\forall k \geq \hat{k}~\forall j$, and these MSDs are stable.
    	    
        \end{proof}
            
        We are now able to bound the current and future iterations' residues and ensure that stable MSDs exist, but the relationship between these two features is currently missing.
        Combining Lemmas~\ref{lem:ResidueBound} and \ref{lem:ConstantExistence} will allow us to establish this, thereby providing a guaranteed minimum number of stable digits for the current and all future approximants.
            
        \begin{theorem}[Inference of digit stability]
            
            If $\boldsymbol{x}\iter{\hat{k} - 1}$ and $\boldsymbol{x}\iter{\hat{k}}$ share a minimum of $D$ identical MSDs, then $x_j^*$ is consistent with the $D + \floor*{\log_r{\frac{1 - \infnorm{\boldsymbol{G}}}{2\infnorm{\boldsymbol{G}}^{k - \hat{k} + 1}}}} - 1$ MSDs of $x_j\iter{k}~\forall k \geq \hat{k}~\forall j$, and these MSDs are stable.

        \label{thm:StableDigits}
    	\end{theorem}
        
        \begin{proof}
            
            Since the $D$ MSDs of each element of approximants $\hat{k} - 1$ and $\hat{k}$ are identical, we can apply Lemma~\ref{lem:ResidueBound} to approximants $\hat{k} - 1$ and $\hat{k}$ to find that \eqref{eqn:res_bound_G} holds for $\hat{k}$, i.e.,
            \begin{equation*}
                \infnorm{\boldsymbol{s}\iter{\hat{k}}} < \frac{2\infnorm{\boldsymbol{G}}}{1 - \infnorm{\boldsymbol{G}}}r^{-D}.
            \end{equation*}
            Substituting this inequality into \eqref{eqn:residue_bound_k}, we can deduce that
            \begin{align*}
                \infnorm{\boldsymbol{s}\iter{k}} &< \infnorm{\boldsymbol{G}}^{k - \hat{k}}\frac{2\infnorm{\boldsymbol{G}}}{1 - \infnorm{\boldsymbol{G}}} r^{-D}
                \\
                &= r^{-\left(D + \log_r{\frac{1 - \infnorm{\boldsymbol{G}}}{2\infnorm{\boldsymbol{G}}^{k - \hat{k} + 1}}}\right)}
                \\
                &\leq r^{-\left(D + \floor*{\log_r{\frac{1 - \infnorm{\boldsymbol{G}}}{2\infnorm{\boldsymbol{G}}^{k - \hat{k} + 1}}}}\right)}.
            \end{align*}
            We can therefore apply Lemma~\ref{lem:ConstantExistence} with this bound on $\infnorm{\boldsymbol{s}\iter{\hat{k}}}$, from which we are finally able to infer that $x_j^*$ is consistent with the $D + \floor*{\log_r{\frac{1 - \infnorm{\boldsymbol{G}}}{2\infnorm{\boldsymbol{G}}^{k - \hat{k} + 1}}}} - 1$ MSDs of $x_j\iter{k}~\forall k \geq \hat{k}~\forall j$, and that those MSDs are stable.
        \end{proof}
        
        Examination of Theorem~\ref{thm:StableDigits} allows us to understand the shapes of the stability regions seen in Fig.~\ref{plt:intro} for different $\infnorm{\boldsymbol{G}}$.
        The relationship between the number of identical MSDs within approximants $\hat{k}-1$ and $\hat{k}$ and the quantity that stabilize by approximant $\hat{k}$ is controlled by $D + \floor*{\log_r{\frac{1 - \infnorm{\boldsymbol{G}}}{2\infnorm{\boldsymbol{G}}^{k - \hat{k} + 1}}}} - 1$ with $k = \hat{k}$, i.e., $D + \floor*{\log_r{\frac{1-\infnorm{\boldsymbol{G}}}{2\infnorm{\boldsymbol{G}}}}} - 1$.
        For the most well conditioned systems, i.e., those with low $\infnorm{\boldsymbol{G}}$, $\log_r{\frac{1-\infnorm{\boldsymbol{G}}}{2\infnorm{\boldsymbol{G}}}}$ is more positive, while for particularly ill-conditioned systems it is more negative.
        This explains the leftward and rightward shifts present in Fig.~\ref{plt:intro} for high and low values of $\infnorm{\boldsymbol{G}}$, respectively.
        The point at which $D$ identical MSDs infer the presence of $D$ stable digits within approximant $\hat{k}$ occurs when $\infnorm{\boldsymbol{G}} = \frac{1}{2r+1}$.
        Beyond $\hat{k}$, we see a linear increase in $\log_r{\frac{1 - \infnorm{\boldsymbol{G}}}{2\infnorm{\boldsymbol{G}}^{k - \hat{k} + 1}}}$, and therefore in the number of stable digits, with $k$.
        This applies even for the most ill-conditioned systems; an increasing number of digits will therefore always stabilize over time.

    \section{Prototype Implementation}
    \label{sec:impl}
    	
    	In order to evaluate the effectiveness of our proposal, we built a hardware implementation based on our previous work, \architect{}~\cite{li2018digit}, modified to allow the runtime inference, and subsequent avoidance of recalculation, of digits known to have stabilized.
    	As the digits of approximant $k$ are generated, their values are compared on-the-fly with those of previously generated approximant $k-1$, fetched from on-chip memory.
    	Once some $D > 0$ successive MSDs are found to be identical across all pairs of elements $\boldsymbol{x}_j\iter{k-1}$ and $\boldsymbol{x}_j\iter{k}$, we designate $\hat{k} \leftarrow k$ and, for all subsequent approximants, the generation of each approximant's first
    	\begin{equation*}
    	    \psi\iter{k} = D + \floor*{\log_r{\frac{1 - \infnorm{\boldsymbol{G}}}{2\infnorm{\boldsymbol{G}}^{k - \hat{k} + 1}}}} - 1
    	\end{equation*}
    	digits is skipped.
    	Note that we do not need to calculate logarithms or perform exponentiation in hardware.
    	Instead, we can use the more computationally efficient form
    	\begin{equation}
    	    \psi\iter{k} = D + \floor*{\alpha - {\left(k-\hat{k}+1\right)}\beta} - 1,
    	    \label{eqn:psi}
    	\end{equation}
    	where $\alpha = \log_r{\frac{1-\infnorm{\boldsymbol{G}}}{2}}$ and $\beta = \log_r{\infnorm{\boldsymbol{G}}}$ are constants that we precompute and feed in along with $\boldsymbol{A}$, $\boldsymbol{b}$, and $\boldsymbol{x}\iter{0}$.
    	
    	Our prototype was a Jacobi method implementation.
    	Jacobi iterates in the form of \eqref{eqn:iter_exact} with $\boldsymbol{A} \in \mathbb{R}^{N \times N}$ and $\boldsymbol{M} = \func{\text{diag}}{\boldsymbol{A}}$.
        As a toy example, our implementation solved linear systems with matrix size $N = 2$.
    	Its datapath is shown in Fig.~\ref{fig:datapath}, and is identical in structure to that used in our previous work~\cite{li2018digit}, facilitating direct comparison.
    	Like its predecessor, this hardware is capable of arbitrary-accuracy result generation but, by virtue of the novel proposal in this article, it can do so more efficiently by skipping the calculation of MSDs known to have stabilized.
		
		\begin{figure}
			\centering
						\newcommand\busmark[3][west] {
				\draw[line width=1pt] (#2-0.1,#3-0.1) to (#2+0.1,#3+0.1);
				\node[anchor=#1] at (#2,#3) {\scriptsize 2};
			}
			\newcommand\busmarkp[3][west] {
				\draw[line width=1pt] (#2-0.1,#3-0.1) to (#2+0.1,#3+0.1);
				\node[anchor=#1] at (#2,#3) {\scriptsize 6};
			}
			\begin{tikzpicture}[yscale=-1.2]
			
			\begin{scope}[shape=circle, line width=1pt]
			\node[draw] (mult1) at (0,0) {\large$\times$};
			\node[draw] (mult2) at (3,0) {\large$\times$};
			\node[draw] (plus1) at (-0.8,1.7) {\large$+$};
			\node[draw] (plus2) at (2.2,1.7) {\large$+$};
			\end{scope}
			
			\node[draw,line width=1pt, minimum width=1.5cm, minimum height=1.5cm] (ram) at (5.8,0.7) {RAM};
			
			\node[anchor=west] (mc1) at (-2.1,-0.6) {$-\dfrac{a_{01}}{a_{00}}$};
			\node[anchor=west] (pc1) at (-2.1,0.53) {$\dfrac{b_{0}}{a_{00}}$};
			\node[anchor=west] (mc2) at (0.9,-0.6) {$-\dfrac{a_{10}}{a_{11}}$};
			\node[anchor=west] (pc2) at (0.9,0.53) {$\dfrac{b_{1}}{a_{11}}$};
			
			\begin{scope}[line width=0.6pt]
			\draw[-latex] (mult1.south) -- +(0,0.2) -| (plus1.north east);
			\draw[-latex] (mult2.south) -- +(0,0.2) -| (plus2.north east);
			
			\draw[-latex] (mc1.east) -| (mult1.north west);
			\draw[-latex] (pc1.east) -| (plus1.north west);
			\draw[-latex] (mc2.east) -| (mult2.north west);
			\draw[-latex] (pc2.east) -| (plus2.north west);
			
			\draw[-latex] (plus2.south) -- +(0,0.3) -| ([xshift=-2mm]ram.south);
			\draw[-latex] (plus1.south) -- +(0,0.4) -| ([xshift=2mm]ram.south);
			
			\draw[-latex] ([xshift=-2mm]ram.north) -- +(0,-0.7) -| (mult2.north east);
			\draw[-latex] ([xshift=2mm]ram.north) -- +(0,-1.2) -| (mult1.north east);
			\end{scope}
			
			\busmark{6.0}{1.8}
			\busmark{5.6}{1.8}
			\busmark{6.0}{-0.3}
			\busmark{5.6}{-0.3}
			\busmark[south]{2.35}{-0.6}
			\busmark[south]{-0.65}{-0.6}
			\busmark{-1.08}{0.95}
			\busmark{-0.52}{0.95}
			\busmark{1.92}{0.95}
			\busmark{2.48}{0.95}
			
			\node at (4.5, -0.9) {$x_0\iter{k}$};
			\node at (3.4, -1.4) {$x_1\iter{k}$};
			\node at (0.6, 2.1) {$x_0\iter{k+1}$};
			\node at (4.2, 2.0) {$x_1\iter{k+1}$};
			
			\end{tikzpicture}
			\caption{
				Arbitrary-precision two-dimensional Jacobi method benchmark datapath~\cite{li2018digit}.
				Adders and multipliers are radix-2 signed-digit online operators with online delay $\delta_\times = 3$ and $\delta_+ = 2$.
			}
			\label{fig:datapath}
		\end{figure}
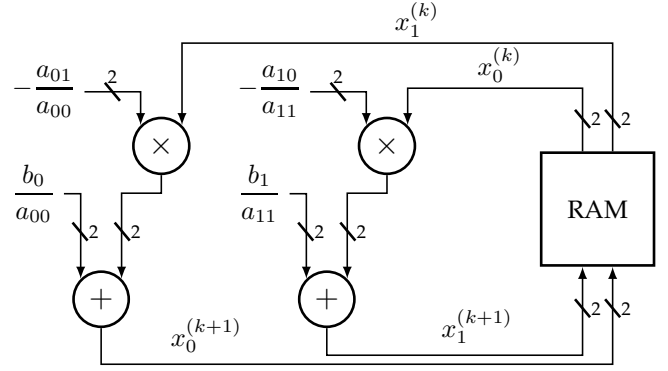
    
    \section{Evaluation}
    \label{sec:eval}
        
        There are three obvious comparison points for our implementation: \architect{} with online delay-based MSD elision~\cite{li2018digit}, the E-method~\cite{ercegovac1977general}, and the broad class of conventional, LSD-first iterative solvers.
        For the MSD-first methods, we conducted theoretical analysis (Section~\ref{subsec:comparison_analysis}) to uncover the shortcomings of the prior art.
        We also performed experiments (Section~\ref{subsec:HWevaluation}) to quantify the gains realized through the employment of our proposal in hardware.
        For comparison against LSD-first arithmetic, we implemented datapaths composed of parallel-in, serial-out (PISO) operators of the same form we previously used to evaluate \architect{}.
        These operate in a similar digit-serial fashion, but require the compile-time determination of precision.
        
        Our hardware implementations all targeted a Xilinx Virtex UltraScale field-programmable gate array (part number XCVU190-FLGB2104-3-E) and were compiled using Vivado~2016.4.
        We verified all results obtained in hardware against golden software models written in MATLAB.
        
        \subsection{Theoretical Analysis}
		\label{subsec:comparison_analysis}
			
			As was mentioned in Section~\ref{plt:intro}, \architect{}'s former MSD elision strategy is unable to infer the existence of stable digits~\cite{li2018digit}.
			In the worst case, as shown in Table~\ref{tab:comparison}, we are forced to compute the values of $\delta$ \emph{more} MSDs for \emph{every} approximant when using that method, potentially wasting significant time and energy in doing so.
			The hardware realization of the proposal in this article is actually simpler than its online delay-based predecessor, leading to the multiple performance boosts we elaborate upon in Section~\ref{subsubsec:perf_comparison}.
			A benefit of our previous proposal is its applicability to any iterative method.
			We leave the generalization of the technique we propose in this article to future work.
			
			\begin{table*}
			    \begin{threeparttable}
    				\caption{Properties of Approaches for the Inference of Identical and Stable MSDs in Current and Future Approximants}
    				\centering
    				\begin{tabular}{ccccccc}
    					\toprule
    					\multirow{2}{*}{Approach}														& \multirow{2}{*}{\makecell{Iterative\\method}}	& \multirow{2}{*}{\makecell{Runtime\\detection}}	& \multirow{2}{*}{$\infnorm{\boldsymbol{G}}$}	& \multirow{2}{*}{$\infnorm{\boldsymbol{b}}$}	& \multirow{2}{*}{\makecell{Guaranteed-stable digits in\\approximant $k \geq \hat{k}$}}						& \multirow{2}{*}{\makecell{Guaranteed-identical digits between\\approximants $k$ and $k + 1~\forall k \geq \hat{k}$}}	\\
    					\\
    					\midrule
    					Our previous work~\cite{li2018digit}	& Any								& \yup								& --			& $\left[0, \infty\right)$						& $0$                              															& $D - \delta{\left(k - \hat{k} + 1\right)}$																\\
    					E-method~\cite{ercegovac1977general}											& Jacobi											& \nope												& $\left[0, \nicefrac{1}{2r}\right]$				& $\left[0, 1\right)$				& $D + k - \hat{k} + 1$																							& $D + k - \hat{k} + 1$																										\\
    					\textbf{This work}																& Stationary										& \yup												& $\left[0, 1\right)$	        				& $\left[0, \infty\right)$						& $D + \floor*{\alpha - {\left(k-\hat{k}+1\right)}\beta} - 1$	& $D + \floor*{\alpha - {\left(k-\hat{k}+1\right)}\beta} - 1$				\\
    					\bottomrule
    				\end{tabular}
    				\begin{tablenotes}
    				    \item\emph{To enable comparison, we assume that $D$ MSDs of all elements of the most recently computed two approximants, $\hat{k}-1$ and $\hat{k}$, are known to be the same.}
        				\item\emph{For compactness, we abbreviate $\alpha = \log_r{\frac{1-\infnorm{\boldsymbol{G}}}{2}}$ and $\beta = \log_r{\infnorm{\boldsymbol{G}}}$ in the final row of the table.}
    				\end{tablenotes}
    				\label{tab:comparison}
    			\end{threeparttable}
			\end{table*}
			
			The E-method, designed for the efficient evaluation of polynomial and rational functions, is the only existing work allowing the declaration of MSDs as stable across the approximants of an iterative algorithm~\cite{ercegovac1977general}.
			Its MSD-first Jacobi solver produces one new less-significant digit for each of the elements of its solution vector per iteration.
			To achieve this, the target linear system $\boldsymbol{A}\boldsymbol{x} = \boldsymbol{b}$ must fulfill a list of strict conditions.
			In particular:
			\begin{enumerate*}[label=(\roman*)]
                \item $\infnorm{\boldsymbol{G}} \leq \nicefrac{1}{2r}$, i.e., a more restrictive requirement than strict diagonal dominance of $\boldsymbol{A}$, and
                \item $\infnorm{\boldsymbol{b}} < 1$. \label{list:emethod_b}
            \end{enumerate*}
			\ref{list:emethod_b} is required since $\boldsymbol{b}$ forms the algorithm's initial internal residue, which must begin and remain bounded within $\left(-1,1\right)^N$ in order to produce valid digits at each iteration.
			
    		As reflected in Table~\ref{tab:comparison}, our proposal is far less restrictive than the E-method.
    		Our work holds for any stationary iterative method, while the E-method is a particular Jacobi implementation.
    		Furthermore, we impose no restrictions upon the target system beyond $\infnorm{\boldsymbol{G}} < 1$, meaning that users can realize the benefits of digit stability even for very poorly conditioned matrices.
    		In order to achieve the same rate of stable MSD growth, solving \eqref{eqn:psi} for $\beta = 1$ shows that our proposal requires $\infnorm{\boldsymbol{G}} = \nicefrac{1}{r}$: double that for the E-method.
            This technique is thus able to achieve the E-method's growth rate for a wider range of differently conditioned matrices.
    		With $\infnorm{\boldsymbol{G}} < \nicefrac{1}{r}$, we achieve a growth rate faster than the E-method's, while the opposite is true when $\infnorm{\boldsymbol{G}} \in \left(\nicefrac{1}{r},1\right)$.
    		An advantage of the E-method over our proposal is that the former does not require knowledge of MSDs shared between approximants; the conditions enumerated above guarantee that digits will begin to stabilize immediately.
    		However, as we showed in our previous work, it is trivial to implement logic to detect the existence of identical MSDs in successive approximants~\cite{li2018digit}.

	    \subsection{Empirical Analysis}
	    \label{subsec:HWevaluation}
	        
	        In order to compare the performance of our new hardware implementation (Section~\ref{sec:impl}) against that of our previous work~\cite{li2018digit}, we experimented with linear systems of the form
	        \begin{equation}
				\boldsymbol{A}_m =
				\begin{pmatrix} 
					1			&	1 - 2^{-m}	\\
					1 - 2^{-m}	&	1
				\end{pmatrix},
    			\quad		
    			\boldsymbol{b} =
    				\begin{pmatrix}
    					b_0	\\
    					b_1
    				\end{pmatrix},
    			\quad
    			\boldsymbol{x}\iter{0} = \boldsymbol{0},
    			\label{eqn:exp_form}
			\end{equation}
			with $b_0$ and $b_1$ randomly selected from a uniform distribution in the range $\left[0,1\right)$.
			\revision{We used the termination criterion $\twonorm{\boldsymbol{A}_m\boldsymbol{x} - \boldsymbol{b}} < \eta$, with $\eta \in \left(0,1\right]$}.
			The conditioning of $\boldsymbol{A}_m$ was controlled via $m \geq 0$, and convergence was always guaranteed since $\boldsymbol{A}_m$ is strictly diagonally dominant $\forall m$.
			\revision{This setup mirrored that employed in our previous work~\cite{li2018digit}, enabling direct comparison.}
		
		    \subsubsection{Scalablity Comparison}
            \label{subsubsec:perf_comparison}
    		    
                In Fig.~\ref{plt:digit_cc_time}, we consider the scalability of arbitrary-precision two-dimensional Jacobi solvers featuring the techniques enabling the avoidance of MSD recomputation detailed in this article and our previous work~\cite{li2018digit}.
                For these experiments, we fixed $m=1$ in \eqref{eqn:exp_form} and varied accuracy bound $\eta$.
                
        		\begin{figure}
        			\centering
        			\begin{tikzpicture}

	\begin{groupplot}[
		scale only axis,
		width=0.82\columnwidth,
		height=0.51\columnwidth,
		group style={
		    group size=1 by 2,
		    xlabels at=edge bottom,
		    xticklabels at=edge bottom,
		    vertical sep=2em
		},
		xlabel near ticks,
		xlabel={Accuracy bound $\eta$},
		ylabel absolute,
		xmode=log,
		ymode=log,
		xmin=2.5,
		xmax=1400,
		xtick={4,8,16,32,64,128,256,512,1024},
		xticklabels={$2^{-4}$,$2^{-8}$,$2^{-16}$,$2^{-32}$,$2^{-64}$,$2^{-128}$,$2^{-256}$,$2^{-512}$,$2^{-1024}$},
		xticklabel style={rotate=30}
	]

    	\nextgroupplot[
    		ylabel={Solve time (s)},
    		ymin=0.0000001,
    		ymax=0.1,
    		ytick={0.1,0.01,0.001,0.0001,0.00001,0.000001,0.0000001},
    		yticklabels={$10^{-1}$,$10^{-2}$,$10^{-3}$,$10^{-4}$,$10^{-5}$,$10^{-6}$,$10^{-7}$},
    	]
    
    	\addplot[arith plot]
    	table [x=Digits, y=Time_ARITH]{data/digit_time.dat};
    	\addplot[TC_Bri_1 plot]
    	table [x=Digits, y=Time_TC_Bri]{data/digit_time.dat};
    
    	\node [text width=1em,anchor=north] at (axis description cs:0.5,1) {\subcaption{\label{plt:accuracy_time}}};
    	
    	\nextgroupplot[
    		ylabel={Stable digits elided},
    		ymin=1,
    		ymax=1000000,
    		ytick={1, 10, 100,1000,10000,100000, 1000000},
    		yticklabels={$10^{0}$,$10^{1}$,$10^{2}$,$10^{3}$,$10^{4}$,$10^{5}$,$10^{6}$}
    	]
	
		\addplot[bar plot]
		table [x=Digits, y=Digit_Diff]{data/digit_time.dat};
		
		\addplot[TC_Bri_1 plot]
		table [x=Digits, y=DigitTC_Bri]{data/digitTCBri.dat};
        \label{plt:accuracy_new}
		\addplot[arith plot]
		table [x=Digits, y=DigitARITH]{data/digitarith.dat};
		\label{plt:accuracy_arith}
		
		\node [text width=1em,anchor=north] at (axis description cs:0.5,1) {\subcaption{\label{plt:accuracy_digits}}};
	
	\end{groupplot}

\end{tikzpicture}
        			\caption{
        			    How the requested accuracy bound $\eta$ affects the (\subref{plt:accuracy_time})~solve time and (\subref{plt:accuracy_digits})~number of stable digits that do not need to be calculated by \architect{} implementations using the MSD recalculation avoidance strategies introduced in this article~(\ref{plt:accuracy_new}) and our previous work~\cite{li2018digit}~(\ref{plt:accuracy_arith}) for the solution of systems of the form in \eqref{eqn:exp_form} with $m=1$.
        			    The bars in (\subref{plt:accuracy_digits}) denote the absolute differences between the competing implementations.
        			    Points with zero elided stable digits are not visible due to (\subref{plt:accuracy_digits})'s logarithmic $y$-axis.
        			}
        			\label{plt:digit_cc_time}
            	\end{figure}
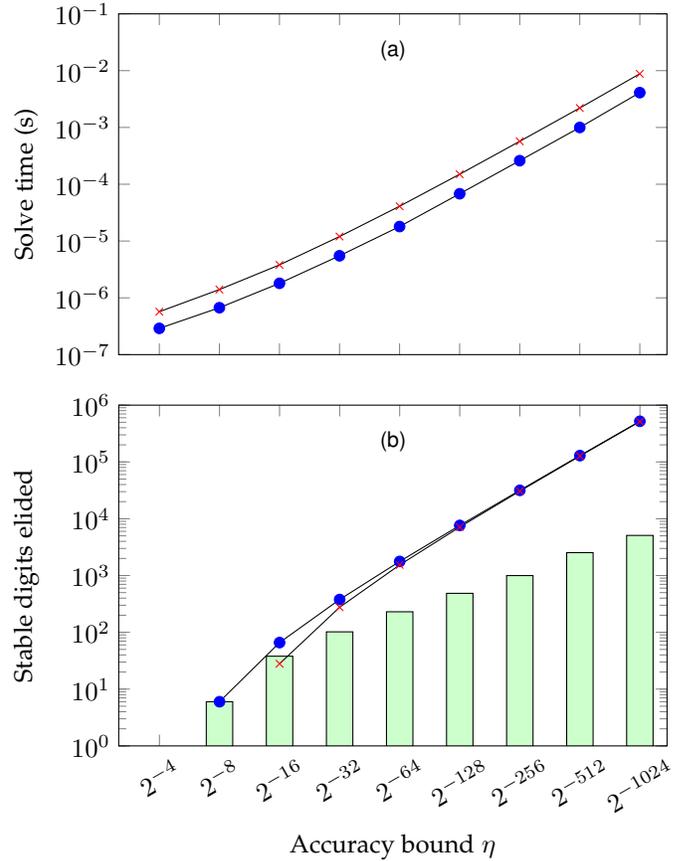
            	
            	Fig.~\ref{plt:accuracy_time} shows that we achieve approximately constant solve time speedups over our previous work.
            	Speedups ranged from 2.0$\times$ (for $\eta = 2^{-4}$) to 2.2$\times$ ($2^{-1024}$).
            	The saturation is due to properties of the arbitrary-precision arithmetic operators shared by both implementations, which require an increasing number of clock cycles to generate each digit as the significance of those digits decreases~\cite{ARCHITECT, ARCHITECT_TVLSI}.
            	As $\eta$ falls, the increasing time per digit generation begins to dominate the gains realized through our new proposal's MSD elision.
            	
        		Fig.~\ref{plt:accuracy_digits} shows that our new analysis allows us to avoid the recomputation of a mean 1.3$\times$ more MSDs than when using the online delay-based proposal introduced in our previous work.
        		With a very low accuracy requirement, $\eta = 2^{-4}$, neither implementation computes for long enough to allow for any MSD elision.
        		Our new proposal becomes effective sooner than its predecessor, at $\eta = 2^{-8}$ rather than $2^{-16}$, due to the former's lack of dependence on online delay $\delta$.
        		For our highest tested accuracy, that with $\eta = 2^{-1024}$, the difference in uncomputed MSDs was 5085 in favor of our new technique.

                Along with the approximately linear increase in newly elided MSDs shown in Fig.~\ref{plt:accuracy_digits}, the speedups shown in Fig.~\ref{plt:accuracy_time} were the result of logic simplifications---and consequently maximum operating frequency increases---over our former implementation.
                The digit generation-scheduling logic for our new implementation is more straightforward than that of its predecessor due to the latter's aforementioned dependence on $\delta$.
                As shown in Table~\ref{tab:HW_comparison}, the implementation we propose in this article is smaller and faster than that using our formerly proposed MSD elision approach.
        		
    			\begin{table}
    				\caption{Comparison of Iterative Solvers with MSD Elision Capability}
    				\centering
    					\begin{tabular}{ccccc}
    						\toprule
    						\multirow{2}{*}{Approach}													& \multirow{2}{*}{\makecell{Lookup\\tables}}	& \multirow{2}{*}{Flip-flops}	& \multirow{2}{*}{\makecell{Memory\\blocks}}	& \multirow{2}{*}{\makecell{Max. operating\\frequency (MHz)}}	\\
    						\\
							\midrule
    						\multirow{2}{*}{\makecell{Our previous\\work~\cite{li2018digit}}}	& \multirow{2}{*}{1191}							& \multirow{2}{*}{992}			& \multirow{2}{*}{24}						& \multirow{2}{*}{150}											\\
    						\\
    						\textbf{This work}															& 1047											& 849							& 24										& 190															\\
    						\bottomrule
    					\end{tabular}
    				\label{tab:HW_comparison}
    			\end{table}
    
    		\subsubsection{Performance Comparison}
    		\label{subsubsec:qual_comparison}
                
                We now show how the conditioning of $\boldsymbol{A}_m$ affects the performance of our arbitrary-precision iterative solvers compared to implementations relying on traditional LSD-first arithmetic.
                For the experiments reported in Fig.~\ref{plt:qual}, we relaxed the constraint on $m$ in \eqref{eqn:exp_form} but fixed $\eta = 2^{-6}$.
    	        
                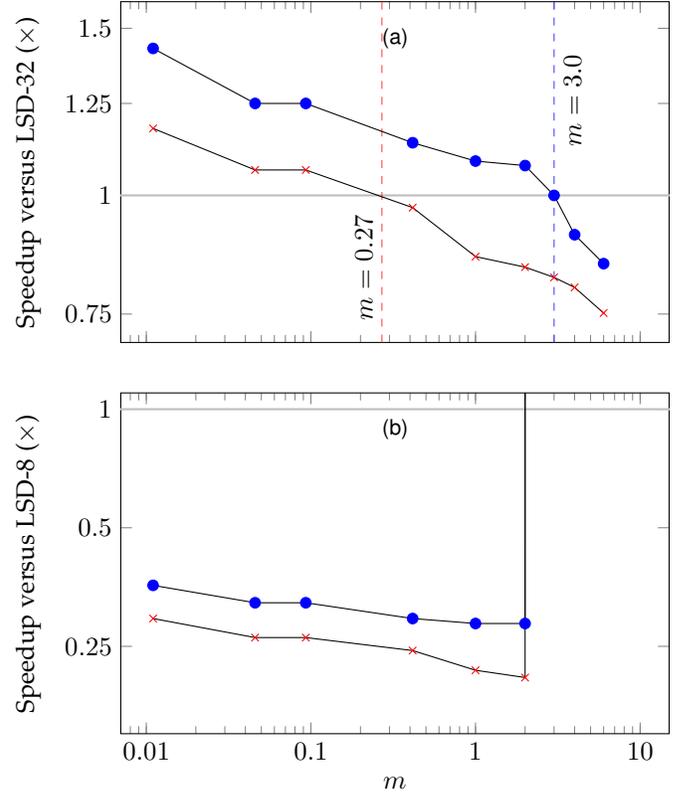
\begin{figure}
        			\centering
        			\begin{tikzpicture}
	
	\begin{groupplot}[
		scale only axis,
		width=0.82\columnwidth,
		height=0.51\columnwidth,
		group style={
		    group size=1 by 2,
		    xlabels at=edge bottom,
		    xticklabels at=edge bottom,
		    vertical sep=2em
		},
		xlabel near ticks,
		xlabel={$m$},
		ylabel absolute,
		xmode=log,
		ymode=log,
		log ticks with fixed point,
		xmin=0.007,
		xmax=15
	]

    	\nextgroupplot[
    		ylabel={Speedup versus LSD-32 ($\times$)},
    		ytick={0.1,0.25,0.5,0.75,1,1.25,1.5},
    		ymin=0.7,
    		ymax=1.6,
    	]
    		
    	\addplot[thick, black!25] coordinates {(0.001,1) (50,1)};
    	\addplot[dashed, red!75] coordinates {(0.27,1.6) (0.27,0.7)};
    	\addplot[dashed, blue!75] coordinates {(3.0,1.6) (3.0,0.7)};
    	\node[anchor=north,rotate=90] at (axis cs:0.165,0.84) {$m = 0.27$};
    	\node[anchor=north,rotate=90] at (axis cs:3.05,1.26) {$m = 3.0$};
    
    	\addplot[TC_Bri_1 plot]
    	table [x=m, y expr={1/\thisrow{TCB_AB}}]{data/qual_AB.dat};
    	\label{plt:qual_TCB}
    
    	\addplot[arith plot]
    	table [x=m, y expr={1/\thisrow{Arith_AB}}]{data/qual_AB.dat};
    	\label{plt:qual_arith}
    	
    	\node [text width=1em,anchor=north] at (axis description cs:0.5,1) {\subcaption{\label{plt:qual_32}}};
    
    	\nextgroupplot[
    		ylabel={Speedup versus LSD-8 ($\times$)},
    		ytick={0.025,0.05,0.1,0.25,0.5,1},
    		ymin=0.15,
    		ymax=1.1
    	]
    
    	\addplot[thick, black!25] coordinates {(0.001, 1) (50, 1)};
    
    	\addplot[TC_Bri_1 plot]
    	table [x=m, y expr={1/\thisrow{TCB_BC}}]{data/qual_BC.dat};
    
    	\addplot[arith plot]
    	table [x=m, y expr={1/\thisrow{Arith_BC}}]{data/qual_BC.dat};
    	
    	\node [text width=1em,anchor=north] at (axis description cs:0.5,1) {\subcaption{\label{plt:qual_8}}};

	\end{groupplot}
	
\end{tikzpicture}
        			\caption{
        				How the conditioning of $\boldsymbol{A}_m$ affects the solve time of \architect{} implementations with MSD elision implemented per the proposal in this article~(\ref{plt:qual_TCB}) and our previous work~\cite{li2018digit}~(\ref{plt:qual_arith}) versus LSD-first arithmetic with a fixed precision of (\subref{plt:qual_32})~32 and (\subref{plt:qual_8})~8 bits.
        			    As a result of the analysis presented herein, our new implementation computes more quickly than LSD-32 when $m \leq 3.0$, whereas our previous implementation can only beat LSD-32 when $m \leq 0.27$.
        				(\subref{plt:qual_8}) shows that both arbitrary-precision iterative solvers lead to an effectively infinite speedup when $m > 2$ since LSD-8 cannot ever converge to accurate-enough results.
        				While performance slowdowns were observed for $m \leq 2$, our new proposal outperformed its predecessor in all cases, as for LSD-32.
        			}
        			\label{plt:qual}
        		\end{figure}

                In Fig.~\ref{plt:qual_32}, we compare our implementations against a Jacobi solver featuring LSD-first PISO arithmetic operators with a precision of 32 bits (LSD-32), a commonly encountered data width.
                For the solution of well conditioned linear systems, i.e., those with low $m$, LSD-32 is said to have \emph{over-budgeted} precision: results take longer to compute than had a lower precision been chosen instead.
                As a result, both \architect{}-based implementations compute more quickly than LSD-32 when $m \leq 0.27$.
                The benefits of our new MSD elision strategy come to the fore with higher $m$.
                For $0.27 < m \leq 3.0$, our new implementation beats its LSD-first competitor in terms of solve time, while that presented in our previous work does not.

                Fig.~\ref{plt:qual_8} shows the results of the same experiments as performed for Fig.~\ref{plt:qual_32}, but compared against an 8-bit LSD-first arithmetic implementation (LSD-8) instead.
                Here, high $m$ results in ill-conditioned systems, for which LSD-8 is said to have \emph{under-budgeted} precision.
                When $m > 2$, only our arbitrary-precision solvers can converge to results of great-enough accuracy.
                In these cases, their performance speedups are effectively infinite.
                For $m \leq 2$, while both our new and prior implementations experience slowdowns versus LSD-8, the former is faster than the latter in all cases.
			
	\section{Conclusion \& Future Work}
		
		In this article, we presented a theorem allowing us to predict the rate of stable MSD growth across the approximants of any stationary iterative method using maximally redundant number representation.
		With knowledge that some number of MSDs are common to two successive approximants, our analysis allows us to declare when, and which, MSDs in all future approximants will stabilize.
		The recomputation of these digits can thus be avoided, facilitating performance speedups.
	    Unlike the E-method, this proposal holds, and is of benefit for, linear systems of any conditioning.
		
		We demonstrated efficiency over our previous work~\cite{li2018digit} and conventional (LSD-first) arithmetic implementations using a hardware implementation of our proposal for the Jacobi method.
		Against the former, we achieved speedups of 2.0--2.2$\times$ for the solution of a range of representative two-dimensional linear systems.
		Versus the latter, we demonstrated gains in cases where LSD-first solvers have precisions either too low or too high to suit the problems at hand.

        In the future, we will extend our analysis to more iterative methods, including gradient descent and Krylov subspace methods.
        We foresee that MSD-first stochastic gradient descent with digit stability declaration would be of particular interest to the deep learning community.
        We are also keen to adapt our proposal to Newton's method, for which we expect to achieve substantial performance gains due to its quadratic convergence.
        
	\section*{Acknowledgments}
	
	    The authors are grateful for the support of the United Kingdom EPSRC (grants EP/P010040/1 and EP/L016796/1), Imagination Technologies, the Royal Academy of Engineering, and the China Scholarship Council.
	    They also wish to thank Milos D. Ercegovac for his helpful suggestions.
	    
	    Supporting data for this article are available online at https://doi.org/10.5281/zenodo.3564471.

    \IEEEtriggeratref{7}
    
	\bibliographystyle{IEEEtran}
	\bibliography{ARITH}

\end{document}